\documentclass[a4paper,12pt]{article}

\usepackage{fullpage}

\usepackage{graphicx}

\usepackage{amsbsy}
\usepackage{latexsym}
\usepackage{amsfonts}
\usepackage{amssymb}
\usepackage[usenames]{color}
\usepackage{amsmath,amsthm}

\let\oldproofname=\proofname
\renewcommand{\proofname}{\upshape\bfseries{\oldproofname}}

\usepackage{tikz}
\usetikzlibrary{arrows}
\usepackage{mathdots}

\DeclareMathOperator{\linspan}{span}

\DeclareMathOperator{\supp}{supp}

\providecommand{\abs}[1]{\lvert#1\rvert}
\providecommand{\bigabs}[1]{\bigl\lvert#1\bigr\rvert}
\providecommand{\Bigabs}[1]{\Bigl\lvert#1\Bigr\rvert}
\providecommand{\biggabs}[1]{\biggl\lvert#1\biggr\rvert}

\providecommand{\norm}[1]{\lVert#1\rVert}

\providecommand{\Bignorm}[1]{\Bigl\lVert#1\Bigr\rVert}

\providecommand{\ceil}[1]{\lceil#1\rceil}

\newcommand{\be}{\begin{equation}}
\newcommand{\ee}{\end{equation}}

\newtheorem{theorem}{Theorem}[section]
\newtheorem{lemma}[theorem]{Lemma}

\newtheorem{cor}[theorem]{Corollary}

\theoremstyle{definition}

\newtheorem{remark}[theorem]{Remark}

\newcommand{\cN}{{\mathcal{N}}}
\newcommand{\cF}{{\mathcal{F}}}
\newcommand{\cH}{{\mathcal{H}}}
\newcommand{\cC}{\mathcal{C}}
\newcommand{\cI}{\mathcal{I}}

\newcommand{\e}{\varepsilon}
\newcommand{\wh}{\widehat}

\newcommand{\T}{{\mathbb{T}}}

\newcommand{\Chi}{\raise .3ex
\hbox{\large $\chi$}} 

\newcommand{\vp}{\varphi}

\newcommand{\R}{\mathbb{R}}
\newcommand{\E}{\mathbb{E}}

\newcommand{\N}{\mathbb{N}}
\newcommand{\Z}{\mathbb{Z}}

\newcommand{\divergence}{\operatorname{div}}

\newcommand\diam{\mathop{\rm diam}}

\newcommand{\iref}[1]{\eqref{#1}}

\newcommand{\bp}{b_{\rm p}}
\newcommand{\kp}{k_{\rm p}}
\newcommand{\kt}{k_{\rm t}}

\newcommand{\hatkt}{\hat k_{\rm t}}

\begin{document}

\title
{Representations of Gaussian random fields and approximation
of elliptic PDEs with lognormal coefficients
\thanks{%
Research supported by the European Research Council under grant ERC AdG BREAD.
}
}
\author{ 
Markus Bachmayr, Albert Cohen and Giovanni Migliorati
}

\maketitle
\date{}
\begin{abstract}
Approximation of elliptic PDEs with random diffusion coefficients
typically requires a representation of the diffusion field
in terms of a sequence $y=(y_j)_{j\geq 1}$ of scalar random variables.
One may then apply high-dimensional approximation
methods to the solution map $y\mapsto u(y)$.
Although Karhunen-Lo\`eve representations are commonly used, it was
recently shown, in the relevant case of lognormal diffusion fields, 
that they do not generally yield optimal approximation rates. Motivated
by these results, we construct wavelet-type representations of stationary Gaussian 
random fields defined on bounded domains. 
The size and localization properties of these wavelets are studied, and used to obtain 
polynomial approximation results
for the related elliptic PDE which outperform those achievable 
when using Karhunen-Lo\`eve representations.
Our construction is based on a periodic extension of the random field,
and the expansion on the domain is then obtained by simple restriction. This
makes the approach easily applicable even when the computational
domain of the PDE has a complicated geometry. In particular, we apply this construction
to the class of Gaussian processes
defined by the family of Mat\'ern covariances.
\end{abstract}

\section{Introduction}

\noindent 
Stochastic PDEs are commonly used to model uncertain physical phenomena.
One such model that is used, for instance, in groundwater modeling is
the diffusion equation
\be\label{lognormaldiffusion}
  -\divergence(a \nabla u ) = f \text{ in $D$}, \quad u|_{\partial D} = 0,
\ee
with lognormally distributed coefficient, that is, $a = \exp(b)$,
where $b$ is a centered Gaussian random field defined on the 
computational domain $D\subset \R^d$ (where typically $d=2$ or $d=3$). 

Uncertainty quantification aims at describing 
the statistical properties of the resulting solution $u$, with various
computational objectives: evaluating the mean field $\bar u=\E(u)$,
estimating a plausible $a$ based on some measurement data of the solution, 
describing the law of a scalar quantity of interest $Q(u)$. 

For such objectives,  it is most convenient to represent $b$ in the form of an expansion
\be\label{bexpansion}
   b = \sum_{j \geq 1} y_j \psi_j,
\ee
where $y_j$ are i.i.d.\ $\cN(0,1)$, that is, independent scalar standard Gaussian random variables,
and $\psi_j$ are suitable functions on $D$.
Once such an expansion is given, one may introduce approximations to the solution map
\be
y \mapsto u(y), \quad y=(y_j)_{j\geq 1},
\ee
for example by multivariate polynomials in the variables $y_j$. Such approximations provide a fast
way to evaluate $u(y)$ for any choice of $y$ up to some prescribed accuracy, which is of crucial
help for the above mentioned tasks. Expansions of the form \iref{bexpansion}
can also be of practical use for the fast generation of trajectories, provided that
the $\psi_j$ have simple analytic expressions or are easy to compute numerically.

The Gaussian field $b$ is characterized by its covariance function
 \be
 (x,x')\mapsto K(x,x') = \E(b(x)\,b(x')), \quad x,x'\in D.
 \ee 
 A standard way of obtaining a representation \eqref{bexpansion} 
 is by using the Karhunen-Lo\`eve (KL) basis, that is, the 
 $L^2(D)$-orthonormal eigenfunctions $(\vp_j)_{j\geq 1}$ of the integral operator 
 \be
T: v \mapsto Tv= \int_D K(\cdot ,z)\,v(z)\,dz,
\ee
with corresponding eigenvalues $\lambda_j\geq 0$ arranged in decreasing order.
One then obtains \eqref{bexpansion} by setting
\be
\psi_j:= \sqrt{\lambda_j} \vp_j.
\ee
 The distinguishing feature of this particular choice is that in addition to
 the statistical orthogonality $\E(y_jy_k)=\delta_{j,k}$, the functions $\psi_j$ are orthogonal in $L^2(D)$. 
 However, other expansions of the form \eqref{bexpansion} may also be considered
 if one does not impose the $L^2(D)$-orthogonality. 
 
 As shown in \cite{LP}, general expansions
 of the form \eqref{bexpansion} with i.i.d.\ $\cN(0,1)$ coefficients 
are characterized by the fact that the $\psi_j$ form a \emph{tight frame} of the reproducing 
 kernel Hilbert space (RKHS) $\cH$ induced by $K$. Recall that $\cH$ is the 
 the closure of the finite linear combinations of the functions $K_z:=K(\cdot,z)$
 for $z\in D$, with respect to the norm induced by the inner product $\langle K_x,K_z\rangle_{\cH}:=K(x,z)$,
 see \cite{A} for a general treatment.
 Recall also that a tight frame of a Hilbert space $\cH$ is a complete system that satisfies the identity
 \be
 \sum_{j\geq 1} |\langle g,\psi_j\rangle_\cH |^2=\|g\|_{\cH}^2, \quad g\in \cH.
 \ee
We refer to \cite{Dau} for classical examples
of time-frequency or time-scale frames. In contrast to orthonormal bases, 
such systems may be redundant. The possible redundancy 
in \iref{bexpansion} can be illustrated by following trivial example: if $y_1$ and $y_2$ 
are i.i.d. $\cN(0,1)$ and $\psi$ is a given function, then
$y_1\psi+y_2\psi = z (\sqrt 2 \psi)$ with $z=(y_1+y_2)/\sqrt 2$ also $\cN(0,1)$.

 As an elementary yet useful example of the different possibilities for
expanding a Gaussian process, consider the Brownian bridge on $D=[0,1]$ whose covariance
 is given by $K(x,x')=\min\{x,x'\}-xx'$. On the one hand, the representation based on the KL expansion
is given by the trigonometric functions
\be
\psi_j(x):=\frac {\sqrt 2}{\pi j}\sin(\pi j x), \quad j\geq 1.
\label{KLBB}
\ee
On the other hand, another classical representation of this process is given by the 
Schauder basis, which consists of the hat functions
\be
\psi_j(x):=2^{-\ell/2}\sigma(2^\ell x-k), \quad \ell\geq 0, \quad k=0,\dots,2^\ell-1, \quad j=2^{\ell}+k,
\label{Schauder}
\ee
where $\sigma(x)=(1-|2x-1|)_+$. Both systems are orthonormal bases (and thus tight frames)
of the RKHS, which in this case is $\cH=H^1_0(D)$ endowed with norm $\|v\|_\cH:=\|v'\|_{L^2(D)}$.

Lax-Milgram theory ensures that for each individual $y=(y_j)_{j\geq 1} \in U := \R^\N$
such that $\sum_{j\geq 1} y_j\psi_j$ converges in $L^\infty(D)$, the corresponding solution
$u(y)$ is well defined in $V:=H^1_0(D)$, with a-priori bound
\be
\|u(y)\|_{H^1_0}\leq C \exp\biggl(\Bignorm{\sum_{j\geq 1} y_j\psi_j}_{L^\infty}\biggr), \quad C:=\|f\|_{V'}.
\ee
Sufficient conditions have been established, either in terms of the covariance function $K$
or of the size properties of the $\psi_j$, which ensure that the solution map $y\mapsto u(y)$
belongs to $L^k(U,V,\gamma)$ for all $k<\infty$, see \cite{DS,Ch,G,HS,BCDM}. Here
$L^k(U,V,\gamma)$ is the usual Bochner space 
where $\gamma$ denotes the countable tensor product of the univariate standard Gaussian measure. 

In turn, approximation can be obtained in $L^2(U,V,\gamma)$, corresponding to mean-square convergence, by truncation of the 
tensor product Hermite expansion
\be
u(y) = \sum_{\nu \in \cF} u_\nu H_\nu(y) ,\qquad u_\nu=\int_U u(y)H_\nu(y)d\gamma(y) \in V, \quad H_\nu(y) := \prod_{j\geq 1} H_{\nu_j}(y_j).
\ee
Here we have denoted by $(H_n)_{n\geq 0}$ the sequence of univariate Hermite polynomials
with normalization in $L^2(\R, g(t) dt)$ where $g(t)=\frac {1}{\sqrt {2\pi}} e^{-t^2/2}$,
and by $\cF$ the set of finitely supported sequences $\nu=(\nu_j)_{j\geq 1}$ of non-negative integers.

The best error for a given number $n$ of terms retained in the above expansion
is attained by the so-called best $n$-term Hermite approximation $y\mapsto u_n(y)$
obtained by retaining the indices $\nu$ corresponding to the $n$ largest $\norm{u_\nu}_V$. 
By combining Parseval's identity with Stechkin's lemma \cite{De}, the resulting error can be quantified in terms 
of the $\ell^p$-summability of the sequence $(\norm{u_\nu}_V)_{\nu\in\cF}$ for $p<2$, namely
\be
\|u-u_n\|_{L^2(U,V,\gamma)}\leq Cn^{-s}, \quad s:=\frac 1 p-\frac 1 2, \quad C:=\|(\norm{u_\nu}_V)_{\nu\in\cF}\|_{\ell^p(\cF)}.
\ee
As shown in \cite{BCDM}, this extra summability may depend strongly on the particular representation \eqref{bexpansion} that is used, or in other words, on the choice of coordinates $(y_j)_{j\geq 1}$. Whereas previous results establishing summability properties of $(\norm{u_\nu}_V)_{\nu\in\cF}$ are based only on the summability of $(\norm{\psi_j}_{L^\infty(D)})_{j\geq 1}$, that is, on the absolute sizes of the $\psi_j$, the results in \cite{BCM,BCDM} also take into account the localization properties of $\psi_j$. More specifically, the following is shown in \cite{BCDM}.

\begin{theorem}
\label{bcdmthm}
Let $0<p<2$ and let $q=q(p):=\frac {2p}{2-p}$. Assume that there exists a positive sequence $(\rho_j)_{j\geq 1}$ such that
\be
\sup_{x\in D}\sum_{j\geq 1} \rho_j|\psi_j(x)| <\infty
\label{firstcond}
\ee
and
\be
(\rho_j^{-1})_{j\geq 1} \in \ell^q(\N).
\label{secondcond}
\ee
Then the solution map $y\mapsto u(y)$ belongs to $L^k(U,V,\gamma)$ for all $0\leq k<\infty$. Moreover,
$(\|u_\nu\|_V)_{\nu\in \cF} \in \ell^p(\cF)$.  In particular, best $n$-term Hermite approximations
converge in $L^2(U,V,\gamma)$ with rate $n^{-s}$ for $s=\frac 1 p-\frac 1 2=\frac 1 q$.
\end{theorem}

The above result draws an important distinction between representations using 
globally or locally supported functions $\psi_j$. If nothing is assumed on the supports of $\psi_j$, 
we may only apply the following immediate consequence
of Theorem \ref{bcdmthm}.

\begin{cor}
\label{globcor}
If $(\psi_j)_{j\geq 1}$ is a family of functions with arbitrary support
such that $(\|\psi_j\|_{L^\infty})_{j\geq 1}$ belongs to $\ell^r(\N)$ for some $r<1$,
then best $n$-term Hermite approximations
converge in $L^2(U,V,\gamma)$ with rate $n^{-s}$ for $s=\frac 1 r-1$.
\end{cor}

Consider for example the above mentioned 
Brownian bridge. On the one hand, the KL functions given by \iref{KLBB} are globally supported
with size of order $j^{-1}$. In turn, for any $q<\infty$, there exists no sequence $(\rho_j)_{j\geq 1}$
satisfying \iref{firstcond} and \iref{secondcond}, and therefore the best $n$-term truncation
cannot be proved to converge with any algebraic rate. On the other hand, the Schauder functions
given by \iref{Schauder} have local support properties which allow us to fulfill
 \iref{firstcond} with $\rho_j= j^{s}$ for any $s<\frac 1 2$. 
Hence best $n$-term Hermite truncation based on the Schauder representation
is ensured to converge with rate $n^{-s}$ for any $s<\frac 1 2$.

The same analysis applies to a general process on a bounded domain $D\subset \R^d$
if it admits an expansion \iref{bexpansion} where the $\psi_j$ have wavelet-like
localization properties. In this case, it is convenient to index the basis functions $\psi_j$
by a scale-space index $\lambda$, with $\abs{\lambda}$ denoting the corresponding scale level. 
We denote by $\cI$ the set of these indices, where 
$\#\{ \lambda \in \cI \colon \abs{\lambda} = \ell \} \sim  2^{d\ell}$ for $\ell \geq 0$.

\begin{cor}
\label{wavcor}
Let $(\psi_j)_{j\geq 1}= (\psi_\lambda)_{\lambda \in \cI}$ be a wavelet basis such that for some $\alpha >0$,  
\be
 \sup_{x\in D}  \sum_{\abs{\lambda} = \ell} \abs{\psi_\lambda(x)}   \leq   C 2^{-\alpha \ell} , \quad \ell \geq 0.
 \label{scaledec}
\ee
 If $(\|\psi_\lambda \|_{L^\infty})_{\lambda \in \cI}$ belongs to $\ell^q(\cI)$, which holds for $q>\frac{d}\alpha$,
then the  best $n$-term Hermite approximations
converge in $L^2(U,V,\gamma)$ with rate $n^{-s}$ for all $s< \frac1q < \frac{\alpha}d$.
\end{cor}

This leads to the question whether suitable wavelet-type systems $(\psi_j)_{j\geq 1}$ forming orthonormal bases or tight 
frames of $\mathcal{H}$ can also be found for more general $b$ with given covariance kernel $K$. In view of the
above corollary, we are interested in the decay exponent $\alpha$ that can be ensured in \iref{scaledec}. 
For practical purposes, we also need these systems to either have exact analytic expressions or
be computable by simple and efficient numerical procedures.

We shall focus on \emph{stationary} random fields, that is, covariances of the form
\be
K(x,x') = k(x - x'), \quad x,x'\in D,
\ee
where $k$ is an even function defined over $\R^d$ which is the inverse Fourier transform of a positive measure.
One typical class of examples is given by the family of \emph{Mat\'ern covariances}
\be\label{matern}
    k(x) = \frac{2^{1-\nu}}{\Gamma(\nu)} \biggl( \frac{\sqrt{2\nu}\abs{x}}{\lambda}  \biggr)^\nu K_\nu\biggl( \frac{\sqrt{2\nu}\abs{x}}{\lambda} \biggr),
\ee
where $\nu,\lambda > 0$ and $K_\nu$ is the modified Bessel function of the second kind, with Fourier transform given by
\be\label{maternfourier}
  \hat k(\omega) = c_{\nu,\lambda}\, \biggl( \frac{2\nu}{\lambda^2} + \abs{\omega}^2  \biggr)^{-(\nu + d/2)}, \quad c_{\nu,\lambda} := \frac{ 2^d \pi^{d/2} \Gamma(\nu + d/2) (2\nu)^\nu}{ \Gamma(\nu) \lambda^{2\nu}}.
\ee
Here, for the Fourier transform, we use the convention
\be
    \hat f(\omega) = \int_{\R^d} f(x)\,e^{-i x\cdot \omega} \,dx.
\ee
The parameters $\nu$ and $\lambda$ quantify the smoothness and correlation length, respectively, of the process.

The idea of using wavelet-type systems for the representation of Gaussian processes was 
put forward in the pioneering works of Ciesielski \cite{Cie}, motivated by the problem of
analyzing the H\"older smoothness of univariate Gaussian processes. This program was 
pursued in \cite{CKR} where Sobolev-Besov smoothness was investigated, in particular
for the fractional Brownian motion. These papers were based on representing the processes
of interest in the Schauder basis, with resulting components that are generally not independent and
therefore not of the form \iref{bexpansion}.

A general approach was proposed in \cite{BJR} to obtain wavelet-type representations of Gaussian processes 
defined on $\R^d$, with independent components. This approach requires that the covariance
function is the Green's function of an operator of pseudo-differential type, and it includes 
Mat\'ern covariances, see also \cite{CI}. A related approach was proposed in \cite{EHM} for stationary Gaussian
processes, motivated by fast methods for generating trajectories. Both approaches 
strongly rely on the Fourier transform over $\R^d$, and do not carry over in a simple manner to the case of a bounded domain $D\subset \R^d$.

A general construction of wavelet expansions of the form \iref{bexpansion}
for Gaussian processes was recently proposed in \cite{KOPP}, in the general framework of
Dirichlet spaces. Here the needed assumptions are that the covariance operator commutes
with the operator that defines the Dirichlet structure. The latter does not have a simple
explicit form in the case of Mat\'ern covariances on a domain $D$, which makes this approach
difficult to analyze and implement in our setting. Let us also mention the construction in \cite{G2},
where an orthonormal basis for the RKHS is built by a direct Gram-Schmidt process, however,
with generally no size and localization bounds on the resulting basis functions.

In the present paper, we propose an approach where the complications that may arise due to the geometry of $D$ 
are circumvented by performing a periodic continuation of the random field $b$ on a larger torus $\T$.
The existence of such a continuation is discussed in \S 2. This leads to a simple construction
of KL-type and wavelet-type expansions, by restrictions to $D$ of similar expansions defined on $\T$.
While the systems introduced on $\T$ are bases, their restriction to $D$
are redundant frames of the RKHS, the amount of redundancy being essentially reflected by 
the ratio $|\T| / |D|$. Our construction thus
achieves numerical simplicity at the price of redundancy.

The periodic continuation has some parallels to \emph{circulant embedding}, proposed independently in \cite{DN} and \cite{WC} as an algebraic technique for evaluating a stationary Gaussian random field, given by $k$, at the points of a uniform grid. Here, the (block) Toeplitz matrix formed by the grid values of $k$ is embedded into a (block) circulant matrix, enabling a factorization by fast Fourier transform. Although it was shown in \cite{DMS} that a positive definite Toeplitz matrix can always be embedded into a sufficiently large positive definite circulant matrix, its required size depends on the particular $k$ under consideration, similarly to the size of $\T$ in our construction. Under more restrictive assumptions
on $k$, simple procedures produce an embedding into a matrix of size proportional to the original one \cite{DN}.
This strategy has also been applied in the numerical treatment of lognormal diffusion problems by quasi-Monte Carlo (QMC) methods \cite{GKNSSSjcp}. 

Also in the setting of sampling-based methods such as QMC, our results have several potentially advantageous features. 
They yield a periodically extended random field on all of $D$, rather than on a uniform spatial grid. Any further approximation (e.g.\ on unstructured finite element meshes) can thus be adjusted independently of the periodic extension and of a chosen expansion of the random field.
Furthermore, the decay properties of the KL eigenvalues of the periodic process are directly controlled by the decay of the Fourier transform of $k$.

In the case of KL-type expansions, which are studied in \S 3, the functions $\psi_j$ that we obtain are simply the restrictions to $D$ of 
trigonometric functions. This is an advantage in term of numerical simplicity compared to the 
$L^2(D)$-orthogonal KL functions of $b$ which may not be easy to compute accurately,
and in addition may not satisfy uniform $L^\infty$ bounds. Wavelet-type expansions are defined in \S 4, 
and we establish their localization and size properties. In the case of Mat\'ern covariances,
they correspond to the value $\alpha=\nu$ in \iref{scaledec}, where $\nu$ is the smoothness parameter
in \iref{maternfourier}. Therefore, corresponding best $n$-term Hermite approximations converge with algebraic rate $n^{-s}$
for any $0<s<\nu/d$. Finally, in \S 5 details are given on numerical procedures which may be applied 
to define the periodic continuation and to construct the resulting wavelets,
in the case of Mat\'ern covariance, depending on the parameters $\lambda,\nu$.

\section{Periodic continuation of a stationary process}

Let $(b(x))_{x\in\R^d}$ be a real valued, stationary and centered Gaussian process defined on $\R^d$, whose covariance is thus of the form
\be
\E(b(x)b(x'))=k(x-x'),
\ee
where $k$ is a real valued and even function which is the inverse Fourier transform of a non-negative measure $\hat k$. We work under 
the assumption that $\hat k$ is a function such that
\be
0\leq \hat k(\omega) \leq C(1+|\omega|^2)^{-r},\quad \omega\in\R^d,
\label{deck}
\ee
for some $r>d/2$ and $C>0$. Obviously, the Mat\'ern covariances \iref{matern} satisfy this assumption with $r=\nu+d/2$ and $C$ depending on 
$(d,\lambda,\nu)$. We consider the restricted process $(b(x))_{x\in D}$ defined on the bounded domain $D$ of interest.

We aim for representations of the general form \iref{bexpansion}
where the $y_j$ are i.i.d.\ $\cN(0,1)$ and the $(\psi_j)_{j\geq 1}$ are a given sequence of functions defined on $D$.
As explained in the introduction, one natural choice is $\psi_j=\sqrt{\lambda_j}\varphi_j$,
where $(\varphi_j,\lambda_j)$ are the eigenfunctions and eigenvalues of the covariance operator. However,
this renormalized KL representation may not meet our requirements
due to the possibly global support of the $\psi_j$ and due to the slow decay of their $L^\infty$ norms,
while other representations could be more appropriate.

Our strategy for deriving better representations of the process over $D$
is to view it as the restriction to $D$ of a periodic stationary Gaussian process $\bp$
defined on a suitable larger torus $\T$.  As a consequence, any representation
\be
\bp = \sum_{j \geq 1} y_j \tilde \psi_j,
\ee
with $y_j$ i.i.d.\ $\cN(0,1)$ and $(\tilde \psi_j)_{j\geq 1}$ a given system of functions, yields a representation
\be
b = \sum_{j \geq 1} y_j \psi_j, \quad \psi_j := \tilde \psi_j|_D.
\ee 
The construction of $\bp$ requires additional assumptions on the 
covariance function $k$.

Let $\delta := \diam(D)$, so that in a suitable coordinate system, 
$D$ can be embedded into the box $[-\frac\delta2, \frac\delta2]^d$. 
We want to construct a periodic process $(\bp(x))_{x\in \T}$ 
on a torus $\T = [-\gamma,\gamma]^d$ with $\gamma>\delta$, 
whose restriction on $D$ agrees with $b$,
that is, $\bp|_D \sim b$. This is feasible provided
that we can find an even and $\T$-periodic function $\kp$ which agrees 
with $k$ over $[-\delta,\delta]^d$ and such that the Fourier coefficients
\be
c_n(\kp):= \int_\T \kp(z)\,e^{-i\frac {\pi}{\gamma}n\cdot z}\,dz, \quad n\in\Z^d,
\label{fourierkp}
\ee
are positive. In addition we would like that these Fourier coefficients
have a similar rate of decay as the function $\hat k$, that is,
\be
0\leq c_n(\kp)\leq C(1+|n|^2)^{-r}, \quad n\in\Z^d,
\label{deckp}
\ee
for some $C>0$. Note that $\kp$ generally differs from the periodization
$\sum_{n\in\Z^d} k(\cdot+2\gamma n)$,
which corresponds to a periodic process that does not agree with $b$ on $D$.

One natural way of constructing the function $\kp$ is by 
truncation and periodization: first we choose a sufficiently smooth 
and even cutoff function $\phi\colon \R^d\to\R$,
to be specified further,
such that $\phi|_{[-\delta,\delta]^d} = 1$ and $\phi(x)=0$ for $x\notin [-\kappa,\kappa]^d$
where $\kappa:=2\gamma - \delta$, and define the truncation
\be
\kt(z):=k(z)\,\phi(z).
\ee
We now define $\kp$ as the periodization of $\kt$, that is,
\be
\label{periodizedkernel}
 \kp(z) = \sum_{n\in \Z^d} \kt(z+2\gamma n) .
\ee
Obviously, $\kp$ agrees 
with $k$ over $[-\delta,\delta]^d$, and 
\be
c_n(\kp)=\hatkt\left(\frac \pi \gamma n\right).
\label{cnkp}
\ee
Therefore \iref{deckp} follows if we can establish
\be
0\leq \hatkt(\omega) \leq C(1+|\omega|^2)^{-r},\quad \omega\in\R^d,
\label{deckt}
\ee
for some $C>0$. Since we have
\be
\hatkt=(2\pi)^{-d}\,  \hat k * \hat \phi,
\label{hatktconv}
\ee
it is easily seen that the upper inequality in \iref{deckt}
follows from the upper inequality in \iref{deck}, provided that $\phi$ is chosen sufficiently smooth
such that
\be
 |\hat \phi (\omega)| \leq C(1+|\omega|^2)^{-r},
\label{decphi}
\ee
for some $C>0$. Indeed, combining \iref{hatktconv} with \iref{deck} and \iref{decphi}, we obtain
\be
\begin{aligned}
(2\pi)^d |\hatkt (\omega)| &\leq  \left |\int_{|\xi| \leq |\omega|/2}\hat k(\xi)\, \hat \phi(\omega-\xi)\,d\xi\right |
+\left |\int_{|\xi| \geq |\omega|/2}\hat k(\xi)\, \hat \phi(\omega-\xi)\, d\xi\right |\\
& \leq \|\hat k\|_{L^1} \max_{|\xi|\geq |\omega|/2} |\hat \phi(\xi)| + \|\hat \phi\|_{L^1} \max_{|\xi|\geq |\omega|/2}|\hat k(\xi)|\\
& \leq C(1+|\omega|^2)^{-r}.
\end{aligned}
\label{ktdecayestimate}
\ee
The main problem is to guarantee the lower inequality in \iref{deckt}, that is, the non-negativity of $\hatkt$. 
Note that $\hat \phi$ cannot be non-negative: since $1=\phi(0)=(2\pi)^{-d}\int_{\R^d} \hat\phi(\omega) d\omega$,
the non-negativity of $\hat \phi$ would imply that 
\be
|\phi(x)|=(2\pi)^{-d}\left |\int_{\R^d} \hat\phi(\omega)\,e^{ix \cdot \omega}\, d\omega\right |<1, \quad x\neq 0,
\ee
therefore contradicting the assumption 
$\phi|_{[-\delta,\delta]^d} = 1$.

It follows that for any such $\phi$, the convolution operator
\be
v\mapsto v*\hat \phi,
\ee
does not preserve positivity for all functions $v$. Here, we are only interested in preserving positivity for the
particular function $\hat k$. However, the following result shows that this is in general not feasible only under
the assumption \iref{deck}.

\begin{theorem}
For any $r>d/2$, there exists an even function $k$ that satisfies \iref{deck} and such that for any 
$\phi$ satisfying $\abs{\hat\phi(\omega)} \leq C(1+\abs{\omega}^2)^{-s}$ for some $s>r$, $\phi|_{[-\delta,\delta]^d} = 1$, $\phi(x)=0$ for $x\notin [-\kappa,\kappa]^d$ 
for some $\kappa>\delta$,
the function $\hatkt=(2\pi)^{-d}\,\hat k * \hat \phi$ is not non-negative.
\end{theorem}

\begin{proof}
Let $h$ be a non-negative, smooth, even function on $\R^d$ with $h(0) = 1$ and support contained in the unit ball. 
For $\ell \in \N$, we choose arbitrary but fixed $\omega_\ell \in \R^d$ such that $\abs{\omega_\ell} = 2^\ell$.
We now define $k$ by its Fourier transform as 
\be
  \hat k(\omega) := \sum_{\ell \geq 1} 2^{-2r\ell} \Bigl(  h\bigl(\ell (\omega - \omega_\ell)\bigr)  +  h\bigl(\ell (\omega + \omega_\ell)\bigr)  \Bigr).
\ee
Then clearly, \eqref{deck} is satisfied.  As demonstrated above, there exists $\omega^*\in\R^d$ such that $\hat\phi(\omega^*)<0$.
For $\ell> 1$, consider
\be
  \hatkt(\omega^* + \omega_\ell) = (2\pi)^{-d} \int_{\R^d} \hat\phi(\omega^* - \xi) \,\hat k(\xi+ \omega_\ell )\,d\xi =  (2\pi)^{-d} \bigl( I_1(\ell) + I_2(\ell) \bigr),
\ee
where
\[
\begin{aligned}
I_1(\ell) &:=\int_{\R^d} \hat \phi(\omega^* - \xi)   2^{-2r\ell} h(\ell \xi)     \,d\xi= \ell^{-d} 2^{-2r\ell}\int_{\R^d} \hat\phi(\omega^* - \ell^{-1} \xi) \, h(\xi) \,d\xi ,\\
 I_2(\ell) &:= \int_{\R^d} \hat \phi(\omega^* - \xi) \bigl(  \hat k(\xi + \omega_\ell )   - 2^{-2r\ell} h(\ell \xi)   \bigr)   \,d\xi.
\end{aligned}
\]
On the one hand, 
\be
  \lim_{\ell \to \infty} \ell^d 2^{2r\ell} I_1(\ell)  =  \hat\phi(\omega^*) \int_{\R^d} h(\xi)\,d\xi <0.
\ee
On the other hand, $\hat k \in L^1(\R^d)$ and $\hat k(\xi + \omega_\ell )   - 2^{-2r\ell}  h(\ell \xi)$ vanishes for $\abs{\xi} \leq 2^{\ell-2}$. 
As a consequence, 
\be
  \abs{I_2(\ell)} \leq  \norm{\hat k}_{L^1(\R^d)} \max_{\abs{\xi}  \geq 2^{\ell-2}}  \abs{\hat\phi(\omega^* - \xi)} .
\ee
For $\ell$ such that $2^{\ell-2} > 2\abs{\omega^*}$, we thus have
\be
   \abs{I_2(\ell)} \leq  C  \norm{\hat k}_{L^1(\R^d)}  ( 1+ 2^{2\ell-6})^{-s}.
\ee
Therefore 
\be
  \lim_{\ell \to \infty} \ell^d 2^{2r\ell} |I_2(\ell)|=0.
\ee
As a consequence, $\hatkt(\omega^* + \omega_\ell) < 0$ for sufficiently large $\ell$.
\end{proof}

The above counterexample reveals that further assumptions are needed on the covariance function $k$. 
Specifically, we work under the stronger assumptions
\be\label{ksandwich}
c(1+|\omega|^2)^{-s}\leq \hat k(\omega)\leq C(1+|\omega|^2)^{-r},
\ee
for some $s\geq r >d/2$ and $0<c\leq C$, and 
\be\label{kderivint}
  \lim_{R\to\infty}  \int_{|x|>R}   \abs{  \partial^\alpha k (x) } \,dx = 0, \quad  \abs{\alpha} \leq 2\ceil{s}.
\ee

\begin{remark}
In the case of the Mat\'ern covariance \eqref{matern} with parameters $\nu,\lambda > 0$, the assumption \eqref{ksandwich} holds with $s=r = \nu+d/2$ as a consequence of \eqref{maternfourier}. The assumption \eqref{kderivint} actually holds
for all derivation orders $\alpha$, as a consequence of the exponential decay of the modified Bessel functions of the second kind $K_\nu$ and of their derivatives.
This exponential decay can be seen, for example, from the integral representation
\be
  K_\nu(x) = \int_0^\infty e^{-x \cosh t} \cosh\nu t \, dt,
\ee
see \cite[p.\ 181]{Watson}.
\end{remark}

\begin{theorem}\label{phiexistence}
Let $k$ be an even function such that \eqref{ksandwich} and \eqref{kderivint} hold. Then for $\kappa > \delta$ sufficiently large, there exists $\phi$ satisfying \eqref{decphi}, $\phi|_{[-\delta,\delta]^d} = 1$, and $\phi(x)=0$ for $x\notin [-\kappa,\kappa]^d$ such that $\hatkt = (2\pi)^{-d}\,\hat k * \hat \phi$ is non-negative.
\end{theorem}

\begin{proof}
We first choose a function $\phi_{2\delta} \in C^{2p}(\R^d)$ with $p := \ceil{s}$,
supported on $[-2\delta,2\delta]^d$ and such that ${\phi_{2\delta}}|_{[-\delta,\delta]^d} = 1$. Then for each $\kappa\geq 2\delta$, we define
\be
\phi_\kappa(x)=\phi_{2\delta}(2\delta x/\kappa).
\ee
We thus have, for all $\kappa\geq 2\delta$,
\be\label{phiderivbound}
  \max_{\abs{\alpha}\leq 2p} \norm{\partial^\alpha \phi_\kappa}_{L^\infty} \leq  
 \max_{\abs{\alpha}\leq 2p}  \norm{\partial^\alpha \phi_{2\delta}}_{L^\infty} =: D <\infty.
\ee
Note that $\phi_\kappa$ satisfies  \eqref{decphi} with a constant $C$ that depends on $\kappa$.
For a value of $\kappa$ to be fixed further, we take $\phi:=\phi_\kappa$, and let
 $\theta := 1 - \phi$. Then
\be
  \hatkt(\omega) = \hat k(\omega) - \widehat{k\theta}(\omega) \geq c(1 + \abs{\omega}^2)^{-s} - \widehat{k\theta}(\omega), 
\ee
and 
\be
  \bigabs{\widehat{k\theta}(\omega)}  \leq ( 1 + \abs{\omega}^2)^{-p} \int_{\R^d} \abs{(I - \Delta)^{p} ( k \theta ) } \,dx \le ( 1 + \abs{\omega}^2)^{-s} \int_{\R^d} \abs{(I - \Delta)^{p} ( k \theta ) } \,dx.
 \ee
By repeated application of Leibniz' rule and separately bounding each term, one finds 
\be
\int_{\R^d} \abs{(I - \Delta)^{p} ( k \theta ) } \,dx \leq C(\kappa) :=  (1 + 3d)^{p} D \max_{\abs{\alpha} \leq 2p} \int_{|x|>\kappa/2}   \abs{  \partial^\alpha k (x) } \,dx.
\ee
 Since $\lim_{\kappa \to \infty} C(\kappa) = 0$ by \eqref{kderivint}, it follows that $\hatkt$ is non-negative for $\kappa$ 
 chosen large enough such that $C(\kappa)\leq c$.
\end{proof}

\section{Karhunen-Lo\`eve representations}\label{sec:kl}

Let us recall that the standard Karhunen-Lo\`eve (KL) decomposition of the stationary Gaussian process $b$ has
the form \iref{bexpansion} with
\be
\psi_j=\psi_j^{\rm KL}:= \sqrt { \lambda_j}\, \vp_j,
\ee
where $(\lambda_j)_{j\geq 1}$ is the sequence of positive eigenvalues
of the covariance operator
\be
T: v \mapsto Tv=\int_D k(\cdot-x) \,v(x)\,dx,
\ee
arranged in decreasing order, and $(\vp_j)_{j\geq 1}$ is the associated $L^2(D)$-orthonormal basis of eigenfunctions.

Working under assumptions \eqref{ksandwich} and \eqref{kderivint},
the periodized construction described in the previous section provides us with an alternative decomposition based
on the covariance operator associated to the periodized process $\bp$, that is,
\be
T_{\rm p}: v \mapsto T_{\rm p} v=\int_{\T} \kp(\cdot-x)\, v(x)\, dx,
\ee
with $\T=[-\gamma,\gamma]^d$. The $L^2(\T)$-orthonormal eigenfunctions of this operator are explicitly given by
the trigonometric functions
\be
\theta_n(z):= t_{n_1}(z_1)\cdots t_{n_d}(z_d), \quad n=(n_1,\dots,n_d) \in \N_0^d,
\ee
where $t_0(z)=(2\gamma)^{-1/2}$ and
\be
t_{2m}(z)=\gamma^{-1/2}\cos\left(\frac {m\pi z} \gamma \right) \quad{\rm and }\quad t_{2m-1}(z)=\gamma^{-1/2}\sin\left(\frac {m\pi z} \gamma\right), \quad m\geq 1.
\ee
The eigenvalues are related to the Fourier coefficients $c_n(\kp)$ defined in \iref{fourierkp}. 
For the above eigenfunction $\theta_n$, with each $n_i$ being either of the form
$2m_i$ or $2m_i-1$, the corresponding eigenvalue is
\be
c_m(\kp), \quad m=(m_1,\dots,m_d).
\ee
We denote by $(\lambda_{\mathrm{p}, j})_{j\geq 1}$ a decreasing rearrangement
of these eigenvalues, 
with corresponding eigenfunctions $(\varphi_{\mathrm{p}, j})_{j\geq 1}$. We may thus write 
\be
\bp=\sum_{j\geq 1} y_j \psi_{\mathrm{p}, j},
\ee
where the $y_j$ are i.i.d.\ $\cN(0,1)$ and
\be
\psi_{\mathrm{p}, j}:=\sqrt {\lambda_{\mathrm{p}, j}}\,\varphi_{\mathrm{p}, j}.
\ee
Since $\bp\sim b$ on $D$, this yields a decomposition of $b$ by restriction, that is, taking
\be
\psi_j=\psi_j^{\rm R}:=\sqrt {\lambda_{\mathrm{p}, j}}\,\varphi_{\mathrm{p}, j}|_{D}.
\label{psirest}
\ee

Note that from \iref{deckp} we obtain
a decay estimate of the form
\be
\lambda_{\mathrm{p}, j} \leq C j^{-\frac {2r}d},\quad j\geq 1,
\label{estimdecklper}
\ee
for some $C>0$.  One first observation is that a similar decay estimate holds for 
the original eigenvalues $\lambda_j$.

\begin{theorem}
If the covariance function satisfies \eqref{ksandwich} and \eqref{kderivint}, then we have
\be
\lambda_j \leq C j^{-\frac {2r}d},\quad j\geq 1,
\label{estimdeckl}
\ee
for some $C>0$. 
\end{theorem}

\begin{proof}
We denote by 
\be
V_n = \linspan\{ \varphi_1,\ldots,\varphi_n \},
\ee
the spaces generated by the Karhunen-Lo\`eve
functions. These spaces satisfy the optimality property
\be
  \sum_{j>n} \lambda_j= \E \bigl( \norm{b - P_{V_n} b}^2_{L^2(D)} \bigr)=\min_{\dim(V)=n} \E \bigl( \norm{b - P_V b}^2_{L^2(D)} \bigr),
 \ee
 where $P_V$ is the $L^2(D)$-orthogonal projector. With
 \be
 V_{\mathrm{p},n} := \linspan\{ \varphi_{\mathrm{p},1},\ldots,\varphi_{\mathrm{p},n}\},
 \ee
the spaces generated by the KL functions of $\bp$, we denote by
 \be
 W_n:= \linspan\{ \varphi_{\mathrm{p},1}|_{D},\ldots,\varphi_{\mathrm{p},n}|_{D}\}
 \ee
 their restriction to $D$. We thus have
 \be
  \sum_{j>n} \lambda_j\leq \E \bigl( \norm{b - P_{W_n} b}^2_{L^2(D)} \bigr),
 \ee
 and since $b$ and $\bp$ agree on $D$, it follows that
 \be
  \sum_{j>n} \lambda_j\leq \E \bigl( \norm{\bp - P_{V_{\mathrm{p},n}} \bp}^2_{L^2(\T)} \bigr),
 \ee
 where $P_{V_{\mathrm{p},n}}$ is the $L^2(\T)$-orthogonal projector. Therefore
 \be
  \sum_{j>n} \lambda_j\leq  \sum_{j>n}\lambda_{\mathrm{p}, j} \leq Cn^{1-\frac {2r}d}, \quad n\geq 1,
 \ee
where the second inequality follows from \iref{estimdecklper}. Since the $\lambda_j$ are positive non-increasing,
this implies the decay estimate \iref{estimdeckl}. 
\end{proof}

\begin{remark}
 In the case of the Mat\'ern covariance, in view of \iref{maternfourier},
 we therefore obtain \iref{estimdeckl} with the value $r:=\nu+d/2$.
 This estimate was derived in \cite{GKNSSS} by a different approach, using the theory developed by Widom
 for the eigenvalues of convolution-type operators. This theory makes the assumption that 
 $\hat k$ is unimodal in each variable, see \cite[p.\,290]{W}, which holds for Mat\'ern covariances, but is not 
 needed in the above construction based on assumptions \eqref{ksandwich} and \eqref{kderivint}.
 \end{remark}

One interest of using the representation based on the functions $\psi^{\rm R}_j$ defined by restriction 
according to \iref{psirest} is that the functions $\varphi_{\mathrm{p}, j}$ are explicitly given
by tensorized trigonometric functions. In particular, they are uniformly bounded. It follows that
\be
\|\psi_j^{\rm R}\|_{L^\infty} \leq C\lambda_{\mathrm{p}, j}^{1/2},\quad j\geq 0, \quad 
\ee
and therefore, by \iref{estimdecklper},
\be
\|\psi_j^{\rm R}\|_{L^\infty} \leq Cj^{-\frac{r}d},\quad j\geq 0.
\ee
In contrast, such uniform bounds for the $L^\infty(D)$ norms are generally 
not available for the KL functions $\vp_j$, which are in addition not easily
computable in the case of a general multivariate domain.

In the particular case of Mat\'ern covariances, the $L^\infty$ norms of 
these functions can be estimated by an argument 
introduced in \cite{GKNSSS}, which uses their natural connections with Hilbertian Sobolev spaces. We
briefly recall this argument. Fixing an $s$ such that $\frac{d}2 < s < r=\nu+d/2$, and assuming that the
domain $D$ satisfies the uniform cone condition, we may use Sobolev embedding 
to obtain
\be
\norm{\vp_j}_{L^\infty(D)} \leq C\norm{\vp_j}_{H^s(D)}.
\ee
We then find by interpolation that
\be
 \norm{\vp_j}_{L^\infty(D)}  \leq C\norm{\vp_j}_{L^2(D)}^{1 - s/r}\, \norm{\vp_j}_{H^r(D)}^{s/r}=C\norm{\vp_j}_{H^r(D)}^{s/r} . 
\ee
In order to estimate the $H^r(D)$ norms of the functions $\vp_j$, we
use the following bound for the covariance operator $T$: for any $v\in L^2(D)$,
denoting by $w$ its extension by zero to $\R^d$, we have
\begin{align*}
\|Tv\|_{H^r(D)}^2 & = \| (k*w)|_D\|_{H^r(D)}^2 \leq \|k*w\|_{H^r(\R^d)}^2\\
& = \int_{\R^d}(1+|\omega|^2)^{r} |\hat k(\omega) \hat w(\omega)|^2 d\omega\\
& \leq  C\int_{\R^d}(1+|\omega|^2)^{-r} |\hat w(\omega)|^2 d\omega\\
& \leq  C\int_{\R^d} \hat k(\omega) |\hat w(\omega)|^2 d\omega \\
& = C\langle k*w,w\rangle_{L^2(\R^d)} \leq C \|Tv\|_{L^2(D)}\|v\|_{L^2(D)},
\end{align*}
where we have used the characterization of Hilbertian Sobolev spaces by Fourier transforms and the
particular form of the Mat\'ern covariance. Taking $v=\vp_j$ and using $T\vp_j=\lambda_j\vp_j$,
we thus obtain 
\be
\norm{\vp_j}_{H^r(D)}
\leq C\lambda_j^{-1/2}.
\ee
In summary we have obtained the non-uniform bound
\be
  \norm{\vp_j}_{L^\infty(D)} \leq C \lambda_j^{-\frac {s}{2r}} 
  \label{estvp}
\ee
and therefore, by \iref{estimdeckl},
\be
  \norm{\psi^{\rm KL}_j}_{L^\infty(D)} \leq C j^{-\frac {r-s} d}.
\ee
In particular, we may take $s = \frac{d}{2} + \varepsilon$ for any sufficiently small $\varepsilon>0$ to obtain
\be
   \norm{\psi^{\rm KL}_j}_{L^\infty(D)} \leq C j^{-\frac{r}{d} + \frac12 + \varepsilon}. 
\ee
This needs to be compared with \eqref{estimdecklper}, which in the present case of the Mat\'ern covariance yields
\be\label{psiRdec}
  \norm{\psi^{\rm R}_j}_{L^\infty(D)} \leq C j^{-\frac{r}{d}},
\ee
since $\norm{\vp_{\mathrm{p},j}}_{L^\infty(D)} \leq 1$.

Let us mention that in the particular univariate case $d=1$, numerical experiment seem to indicate that
$\norm{\vp_j}_{L^\infty(D)}$ stays bounded independently of $j$, and therefore that the
upper bound \iref{estvp} is not always sharp. On the other hand, one can also exhibit 
examples of stationary processes such that the corresponding KL functions on the domain $D$
are not uniformly bounded. Take for example the case $D=[-1,1]$ and $k$ such that
$\hat k=\Chi_{[-F,F]}$ for which the KL functions $\vp_j$ coincide with  the univariate prolate
spheroidal functions introduced in \cite{Sl}. It is known that these functions are uniformly
close to the Legendre polynomials $L_j$ as $j\to \infty$, which shows that 
$\norm{\vp_j}_{L^\infty(D)}\sim j^{1/2}$, see \cite{BK}.

In summary, we have obtained substantially better bounds on the decay of $\norm{\psi^{\rm R}_j}_{L^\infty(D)}$ than available for $\norm{\psi^{\rm KL}_j}_{L^\infty(D)}$,
and in addition the $\psi^{\rm R}_j$ are easily computed numerically while this is generally not the case for the $\psi^{\rm KL}_j$.
However, \eqref{psiRdec} still leads to rather severe restrictions on the values of $r$ for which Theorem \ref{bcdmthm} is applicable via Corollary \ref{globcor},
due to the global supports of the functions $\psi^{\rm R}_j$.
In the following section we consider an alternative wavelet-type construction for which Theorem \ref{bcdmthm}, with Corollary \ref{wavcor}, yields an approximation rate for corresponding solutions of $u$ for {any} $r>\frac d 2$.

\section{Wavelet representations}\label{sec:wavelets}

Our starting point is an $L^2(\R)$-orthonormal wavelet basis, that is a basis of the form
\be
\{\vp(\cdot-n)\}_{n\in\Z} \cup\{ 2^{\ell/2}\psi(2^\ell \cdot-n)\}_{\ell\geq 0,n\in\Z}
\ee
where $\vp$ and $\psi$ the scaling function and mother wavelet, respectively. For simplicity
we use the Meyer wavelets, whose construction is detailed in \cite{Dau,Me}, and for which
\be
\supp(\hat \vp)=\biggl[-\frac{4\pi}{3},\frac{4\pi}{3}\biggr]\quad {\rm and}\quad \supp(\hat \psi)=\biggl[-\frac{8\pi}{3}, -\frac{2\pi}{3}\biggr]\cup \biggl[\frac{2\pi}{3},\frac{8\pi}{3}\biggr].
\ee
The functions $\hat \vp$ and $\hat \psi$ may be chosen to be smooth, but for our purpose it will be enough to
assume that they have $M$ uniformly bounded derivatives with an integer $ M \geq d+1$.

Denoting $\psi_0 := \varphi$ and $\psi_1 := \psi$, the multivariate scaling function and wavelets are defined by
\be
 \Phi(x) := \varphi(x_1)\cdots\varphi(x_d),\qquad \Psi_\varepsilon(x) := \psi_{\varepsilon_1}(x_1)\cdots\psi_{\varepsilon_d}(x_d), \quad 
  \varepsilon \in \cC,
\ee
where $\cC:=\{0,1\}^d\setminus \{(0,\ldots,0)\}$. Then 
\be
 \{ \Phi(\cdot - n)\colon n\in\Z^d\} \cup \{  \Psi_{\e,n,\ell} \colon n\in\Z^d , \ell \geq 0, \varepsilon\in\cC \},
\ee
is an orthonormal basis of $L^2(\R^d)$, where we have used the notation
\be
\Psi_{\e,n,\ell}:=2^{d\ell/2} \Psi_\varepsilon (2^\ell \cdot -n).
\ee 
We obtain an orthonormal basis of $L^2(\T)$, by rescaling
and periodization. This basis consists of the constant scaling function
\be
\Phi^{\rm p}(x) := \sum_{m\in\Z^d} (2\gamma)^{-d/2} \Phi\bigl((2\gamma)^{-1}x - m\bigr) = (2\gamma)^{-d/2},
\ee
and the $\T$-periodic wavelets
\be
 \Psi^{\rm p}_{\varepsilon,\ell,n} (x) := 
 \sum_{m\in\Z^d} (2\gamma)^{-d/2} \Psi_{\e,n,\ell} \bigl((2\gamma)^{-1} x -  m \bigr),
\ee
for $n\in \{0,\dots,2^{\ell}-1\}^d, \;\ell\geq 0,\; \e\in\cC$. From the Poisson summation formula, and the support 
properties of $\hat \vp$ and $\hat \psi$, it is 
easily seen that the above wavelets, at a given scale level $\ell$, are finite linear combinations of the Fourier exponentials $e_n$ 
with $\|n\|_\infty \leq 2^{\ell +2}$. In other words, they are trigonometric polynomials of degree at most $2^{\ell+2}$ in each variable.

We now make the following general observation: let $(g_j)_{j\geq 1}$ be any orthonormal basis of $L^2(\T)$, with each basis function
having the Fourier expansion
\be
g_j=(2\gamma)^{-d/2}\sum_{n\in\Z^d} c_n(g_j)\,e_n,
\ee
where
\be
e_n(z):=  (2\gamma)^{-d/2}  e^{i\frac {\pi}{\gamma}n\cdot z},\quad n\in\Z^d.
\ee
Then, defining the functions $(\bar g_j)_{j\geq 1}$ by $\bar g_j:=Sg_j$, where $S$ is
the filtering operator
\be
v\mapsto Sv:=(2\gamma)^{-d/2} \sum_{n\in\Z^d} \sqrt{c_n (\kp)} \,c_n(v)\,e_n,
\label{gjbar}
\ee
we obtain a decomposition
\be
\bp = \sum_{j\geq 1} y_j \bar g_j,
\ee 
where the $y_j$ are i.i.d.\ $\cN(0,1)$, and therefore $b = \sum_{j\geq 1} y_j \bar g_j|_{D}$.

We apply this procedure to the above described periodic wavelet basis,
therefore obtaining new periodic functions 
\be
\bar \Phi^{\rm p}=S\Phi^{\rm p} \quad {\rm and} \quad \bar \Psi^{\rm p}_{\varepsilon,\ell,n}:=S\Psi^{\rm p}_{\varepsilon,\ell,n},
\ee
which are adapted to the decomposition of $\bp$. By construction $\bar \Phi^{\rm p}$ is again a 
constant function with value $(2\gamma)^{-d/2}\sqrt{c_0 (\kp)}$,
and the functions $\bar \Psi^{\rm p}_{\varepsilon,\ell,n}$ are trigonometric polynomials of degree 
at most $2^{\ell+2}$ in each variable. We next study in more detail the size and localization
properties of these functions and show that they essentially behave like a wavelet basis
with normalization $2^{\ell (d/2-r)}$ in place of $2^{\ell d/2}$. This geometric decay of the $L^\infty$ norms, combined
with the localization properties, will allow us to apply Theorem \ref{bcdmthm} and Corollary \ref{wavcor}
for any value of $r>\frac d 2$.

Note that since $\kp$ has been obtained by periodizing the truncated function 
$\kt=k\phi$, an equivalent construction of the functions $\bar \Psi^{\rm p}_{\varepsilon,\ell,n}$ 
is obtained by first defining over $\R^d$ the rescaled and filtered wavelets $\bar \Psi_{\varepsilon,\ell,n}$ according to
\be
\widehat{\bar\Psi}_{\varepsilon,\ell,n}(\omega)=\hatkt^{\frac12}(\omega) \,(2\gamma)^{d/2}\widehat{\Psi}_{\varepsilon,\ell,n}(2\gamma \omega),
\label{filter}
\ee
and then applying $\T$-periodisation, that is,
\be\label{Psiperiodization}
\bar \Psi^{\rm p}_{\varepsilon,\ell,n}(x)=\sum_{m\in \Z^d}\bar \Psi_{\varepsilon,\ell,n}(x+2\gamma m).
\ee
We thus focus our attention on the size and localization properties of the functions $\bar \Psi_{\varepsilon,\ell,n}$.

Note that these functions inherit from the wavelet basis the translation invariance structure,
since, at any given scale level $\ell$, the function $\bar \Psi_{\varepsilon,\ell,n}$ are translates 
by $2\gamma 2^{-\ell}n$ of $\bar \Psi_{\varepsilon,\ell,0}$. However, 
they do not inherit the dilation invariance structure, since
the functions $\bar \Psi_{\varepsilon,\ell,0}$
are not obtained by a simple rescaling of $\bar \Psi_{\varepsilon,0,0}$. 
We introduce rescaled functions $F_{\varepsilon,\ell}$ such that
\be\label{PsiF}
\bar \Psi_{\varepsilon,\ell,0}(x)=2^{\ell (d/2-r)}F_{\varepsilon,\ell}(2^\ell x),
\ee
that is, $F_{\varepsilon,\ell}$ is defined by
\be
F_{\varepsilon,\ell}( x)=2^{-\ell (d/2-r)} \bar \Psi_{\varepsilon,\ell,0}(2^{-\ell} x).
\label{defF}
\ee
Our objective is now to show that the functions $F_{\varepsilon,\ell}$ satisfy a
uniform localization estimate, independently of $\e$ and $\ell$.

For this purpose, we require some additional assumptions on the covariance
function, namely that \eqref{ksandwich} holds with $s=r$, and that 
the partial derivatives of $\hat k$ satisfy improved decay estimates
\be
\label{derkdecay}
 \abs{\partial^\alpha \hat k (\omega)} \leq C ( 1+ \abs{\omega}^2)^{-(r+|\alpha|/2)} , \quad |\alpha| \leq M,
\ee   
for an $M \geq d+1$.
It is easily seen that Mat\'ern covariances satisfy such estimates, by straightforward differentiation of \iref{maternfourier}.

\begin{lemma}\label{sqrtkestimate}
Let $k$ satisfy \eqref{ksandwich}, \eqref{kderivint} with $s=r$ and \eqref{derkdecay} with an integer $M \geq d+1$.
Then $\phi$ in Theorem \ref{phiexistence} can be chosen such that
\be
\label{sqrtktdecay}
 \abs{\partial^\alpha \hatkt^{\frac12} (\omega)} \leq C ( 1+ \abs{\omega})^{-(r+|\alpha|)} , \quad 0\leq  |\alpha| \leq M.
\ee   
\end{lemma}

\begin{proof}
We proceed exactly as in the proof of Theorem \ref{phiexistence} to obtain a family of functions $\phi_\kappa \in C^{2p}(\R^d)$ for $\kappa \geq 2 \delta$, supported on $[-2\delta,2\delta]^d$ and such that ${\phi_{2\delta}}|_{[-\delta,\delta]^d} = 1$, and satisfying
\be\label{phiderivboundext}
  \max_{\abs{\alpha}\leq 2p} \norm{\partial^\alpha \phi_\kappa}_{L^\infty} \leq  
 D <\infty
\ee
for some $D>0$, but here with $p := \ceil{r + M/2}$.

As in the proof of Theorem \ref{phiexistence}, we obtain that for $\phi := \phi_\kappa$ with $\kappa$ sufficiently large, there exist $\tilde c,\tilde C>0$ such that
\be
   \tilde c ( 1 + \abs{\omega}^2 )^{-r} \leq \hatkt(\omega) \leq \tilde C ( 1 + \abs{\omega}^2)^{-r}.
   \label{framekt}
\ee 
Note that $\phi$ obtained in this manner satisfies
\be
   \abs{\hat\phi(\omega)} \leq C ( 1 + \abs{\omega}^2)^{-p}.
\ee
Since $\hatkt = (2\pi)^{-d}\, \hat k * \hat\phi$, we have
\be
  \abs{ \partial^\alpha \hatkt(\omega) } = (2\pi)^{-d} \abs{(\partial^\alpha\hat k * \hat\phi)(\omega)} \leq C ( 1 + \abs{\omega}^2 )^{-(r + \abs{\alpha}/2)},
   \quad \abs{\alpha} \leq M,
\ee
by the same argument used for the upper inequality in \iref{deckt}. 

We now turn to \eqref{sqrtktdecay}. For $\alpha = 0$, this simply follows from \iref{derkdecay} combined with
$(1+|\omega|^2)\leq (1+|\omega|)^2$. For $\alpha \neq 0$ such that $\abs{\alpha} \leq M$, we obtain by induction that $\partial^\alpha\hatkt^{\frac12}$ is of the form
\be
     \partial^\alpha\hatkt^{\frac12} = \sum_{m=0}^{\abs{\alpha}}\sum_{\beta_1 + \cdots+\beta_m = \alpha} C_{\beta_1,\ldots,\beta_m} 
           \hatkt^{\frac12 - m} \prod_{\ell=1}^m \partial^{\beta_\ell} \hatkt
\ee
for certain $C_{\beta_1,\ldots,\beta_m} \in \R$. Since for $\beta_1,\ldots,\beta_m$ such that $\beta_1 + \cdots + \beta_m = \alpha$, we have
\be
   \hatkt^{\frac12 - m}(\omega) \prod_{\ell=1}^m \partial^{\beta_\ell} \hatkt(\omega)
     \leq C ( 1 + \abs{\omega}^2 )^{-\frac{r}2 + mr} \prod_{\ell=1}^m ( 1 + \abs{\omega}^2 )^{-(r + \abs{\beta_\ell}/2)},
     \ee
     and thus
     \be
   \hatkt^{\frac12 - m}(\omega) \prod_{\ell=1}^m \partial^{\beta_\ell} \hatkt(\omega)
      \leq C ( 1 + \abs{\omega})^{-(r+\abs{\alpha})},
      \ee
we arrive at \eqref{sqrtktdecay}.
\end{proof}

\begin{theorem}\label{Fdecay}
Let $k$ satisfy \eqref{ksandwich}, \eqref{kderivint} with $s=r$ and \eqref{derkdecay} with an integer $M \geq d+1$, and let $\phi$ be chosen as in Lemma \ref{sqrtkestimate}.
Then the functions $F_{\varepsilon,\ell}$ satisfy 
\be
|F_{\varepsilon,\ell}(x)|\leq C(1+|x|)^{-M},
\label{Fepselldec}
\ee
for some $C>0$ that is independent of $\ell$ and $\e$.
\end{theorem}

\begin{proof}
From its definition \iref{defF} we have
\be
\wh F_{\varepsilon,\ell}(\omega)=2^{\ell (d/2+r)}  \wh{\bar\Psi}_{\varepsilon,\ell,0}(2^{\ell} \omega),
\ee
and therefore by \iref{filter},
\be
\wh F_{\varepsilon,\ell}(\omega)=(2\gamma)^{d/2} 2^{\ell (d/2+r)}  \hatkt^{\frac12}(2^\ell \omega) \,\widehat{\Psi}_{\varepsilon,\ell,0}(2\gamma 2^\ell \omega)=(2\gamma)^{d/2}2^{\ell r} \hatkt^{\frac12}(2^\ell \omega)\,
 \widehat{\Psi}_{\varepsilon}(2\gamma \omega),
\ee
where we have used the scaling relation between $\Psi_{\varepsilon,\ell,0}$ and $\Psi_{\varepsilon}$.
The functions $\wh F_{\varepsilon,\ell}$ are uniformly compactly supported since
\be
|\omega|_\infty \geq \frac {8\pi}{6\gamma} \quad \implies \quad  \widehat{\Psi}_{\varepsilon}(2\gamma \omega)=0.
\ee
Applying partial differentiation for any $\alpha$ such that $|\alpha|\leq M$, and using the multivariate Leibniz formula,
we find that 
\be
\partial^\alpha \wh F_{\varepsilon,\ell}(\omega)=(2\gamma)^{d/2}2^{\ell r}\sum_{\beta \leq \alpha} {\alpha\choose \beta}
\left( 2^{\ell |\beta|}\partial^\beta  \hatkt^{\frac12}(2^\ell \omega)\right) \left( (2\gamma)^{|\alpha|-|\beta|}\partial^{\alpha-\beta}\widehat{\Psi}_{\varepsilon}(2\gamma \omega)\right).
\ee
The second factor $(2\gamma)^{|\alpha|-|\beta|}\partial^{\alpha-\beta}\widehat{\Psi}_{\varepsilon}(2\gamma \omega)$ in each term is uniformly bounded independently of $\omega$ and $\beta$, in view
of the smoothness assumption that we have imposed on $\hat \psi$. As to the first factor, since we only consider
$\frac{2\pi}{6\gamma} \leq |\omega| \leq \sqrt{d} \frac {8\pi}{6\gamma}$, we may use \iref{sqrtktdecay} to conclude that
\be
|2^{\ell |\beta|}\partial^\beta  \hatkt^{\frac12}(2^\ell \omega)|\leq C2^{-r\ell}.
\ee
It follows that the derivatives $\partial^\alpha \wh F_{\varepsilon,\ell}(\omega)$
are uniformly bounded, independently of $\ell$ and $\e$, for all $|\alpha|\leq M$, which implies \iref{Fepselldec}
since they are in addition uniformly compactly supported.
\end{proof}

As we shall show next, Theorem \ref{Fdecay} implies that Corollary \ref{wavcor} can be applied to the wavelet basis defined by \eqref{Psiperiodization}, that is, with the basis in the corollary chosen as the scaling function $\bar\Phi^{\rm p}$ and the wavelets $\bar\Psi^{\rm p}_{\e,\ell,n}$ ordered by increasing scale level.

\begin{cor}\label{corPsipsum}
Under the assumptions of Theorem \ref{Fdecay}, for each $\ell \geq 0$ one has
\be\label{eqPsipsum}
   \sup_{x\in \T}  \sum_{\e \in \cC} \sum_{n \in \{0,\ldots,2^\ell-1\}^d} \abs{\bar\Psi^{\rm p}_{\e,\ell,n}(x)}   \leq   C 2^{-\alpha \ell} 
\ee	
with $\alpha := r - \frac{d}2$ and $C$ independent of $\ell$.
\end{cor}

\begin{proof}
For the summation over $n$ in \eqref{eqPsipsum}, by \eqref{Psiperiodization} and \eqref{PsiF} we obtain	
\[
\begin{aligned}
   \sum_{n \in \{0,\ldots,2^\ell-1\}^d} \abs{\bar\Psi^{\rm p}_{\e,\ell,n}(x)}
     &= 2^{\ell(d/2-r)} \sum_{m \in \Z^d} \sum_{n \in \{0,\ldots,2^\ell-1\}^d} \abs{F_{\e,\ell}\bigl(2^\ell x + 2\gamma ( 2^\ell m - n) \bigr)}   \\
     &= 2^{\ell(d/2-r)} \sum_{k \in \Z^d} \abs{F_{\e,\ell}(2^\ell x + 2\gamma k)}.
\end{aligned}
\]
By Theorem \ref{Fdecay},
\be
 \sum_{k \in \Z^d} \abs{F_{\e,\ell}(2^\ell x + 2\gamma k)} 
  \leq C \sum_{k \in \Z^d} ( 1 + \abs{2^\ell x + 2\gamma k})^{-M}
   \leq C \max_{z \in \T} \sum_{k \in \Z^d} (1 + \abs{z + 2\gamma k})^{-M},
\ee
and the expression on the right is bounded since $M \geq d+1$.
\end{proof}

\begin{remark}
In the estimate \eqref{eqPsipsum}, the precise value of $M$ enters only into the constant $C$. For numerical purposes, however, larger values of $M$ corresponding to stronger spatial localization of the functions $\bar\Psi^{\rm p}_{\e,\ell,n}$ can be advantageous. Note that in the case of the Mat\'ern covariance, if $\phi \in C^\infty$, then Lemma \ref{sqrtkestimate} can be applied for any integer $M\geq d+1$. If in addition, we have $\hat\vp, \hat\psi\in C^\infty$ for the functions generating the Meyer wavelets, then Theorem \ref{Fdecay} can be applied for any such $M$ as well, and the resulting spatial decay of $\bar\Psi^{\rm p}_{\e,\ell,n}$ is faster than any polynomial order.
\end{remark}

\section{Numerical aspects}

We now discuss in more concrete detail the periodic continuation and the resulting KL and wavelet
representations in the case of the Mat\'ern covariances \iref{matern}. In particular we 
discuss how these representations can be efficiently computed by using FFT and show some
numerical examples which reveal the effect of the parameters $\lambda$ and $\nu$ 
of the Mat\'ern covariance. While these computational principles
apply in any dimension, we work in the univariate setting, both for the sake of notational simplicity
and for visualization purposes. Our computational domain is thus an interval
$\left]-\frac\delta2,\frac\delta2\right[$ for some $\delta>0$. Note however that
varying the correlation length parameter $\lambda$ in \iref{matern} amounts to rescaling the interval. Therefore we fix $\delta=1$, that is
\be
D:=\Bigl]-\frac12,\frac12\Bigr[,
\ee
and study the effect of varying $\lambda$.

\subsection{Truncation and positivity}

We need to first choose a $\phi$ satisfying the conditions in Theorem \ref{phiexistence} such that $\kt = k \phi$ has nonnegative Fourier transform.
One option, which yields $\phi \in C^\infty(\R)$, is based on the function $\theta$ defined by 
\be
  \theta(x) := \begin{cases} \exp(-x^{-1}), & x>0, \\  0,  & x \leq 0. \end{cases}
  \label{theta}
\ee  
We then set
\be
   \phi(x) := \frac{ \theta\Bigl(\frac{\kappa - \abs{x}}{\kappa - \delta}\Bigr) }{\theta\Bigl(\frac{\kappa - \abs{x}}{\kappa - \delta}\Bigr) + 
           \theta\Bigl(\frac{\abs{x} - \delta}{\kappa - \delta}\Bigr)}.
\ee
Recall that $\kappa=2\gamma-\delta >\delta$. In view of Theorem \ref{phiexistence}, in order to ensure $\hatkt\geq 0$, it then suffices to take $\gamma$ sufficiently large.

In order to illustrate the dependence of the required value of $\gamma$ on the parameters $\nu$, $\lambda$ of the Mat\'ern covariance, we now describe, for any chosen value of $\gamma$, a simple scheme for approximating $\hatkt$ based on the discrete Fourier transform. Let $L> \kappa$ (so that $\supp \kt \subset [-L,L]$), and let $N= 2^J$ for some $J>0$. We consider the approximation of $\hatkt$ by the trapezoidal rule,
\be\label{trapezoidalapprox}
  \hatkt(\omega) = \int_{-L}^L \kt(x)\, e^{-i\omega x}\,dx \approx h \sum_{n={-N/2}}^{N/2-1} \kt(x_n) e^{-i\omega x_n} =: D_{L,N}(\omega),
\ee
where $h := 2L/N$ and $x_n := nh$.
The sum on the right side can be evaluated by the FFT to obtain the values $D_{L,N}(\omega_k)$ with
\be\label{omegakdef}
   \omega_k := \frac{\pi k}{L}, \quad k = -\frac{N}2,\ldots, \frac{N}2 - 1.
\ee
Thus, by making $N$ and $L$ large we may approximately compute $\hatkt$ 
on an arbitrarily large range of $\omega$ and with arbitrarily fine sampling rate.

Since $\kt$ agrees on $[-L,L]$ with its $2L$-periodic extension, we have
\be
  \kt(x_n) = \sum_{\ell\in\Z} \biggl(\frac1{2L} \int_{-L}^L \kt(y)\, e^{-\pi i y \ell/L} \,dy \biggr) e^{\pi i x_n \ell / L},
\ee
and consequently
\begin{equation}
	D_{L,N}(\omega_k) = h \sum_{n={-N/2}}^{N/2-1} \Bigl( \sum_{\ell\in\Z} \frac{1}{2L} \hatkt(\omega_\ell) \,e^{\pi i x_n \ell / L}  \Bigr) e^{- \pi i x_n k /L}
	 =  \sum_{m\in\Z} \hatkt(\omega_{k + m N}).
\end{equation}
We thus have the error representation
\be
  \abs{\hatkt(\omega_k) - D_{L,N}(\omega_k)}
   = \Bigabs{ \sum_{\substack{m\in\Z\\ m\neq 0}} \hatkt(\omega_{k+mN}) }.
\ee
In view of \eqref{maternfourier} and \eqref{ktdecayestimate}, in our present setting we obtain 
\be
   \abs{\hatkt(\omega_k) - D_{L,N}(\omega_k)} \leq C N^{-(2\nu + 1)}, \quad k = -\frac{N}2,\ldots, \frac{N}2 - 1,
   \label{estimkt}
\ee
where $C>0$ is independent of $L$ and $N$.

Based on this approximation, we can check positivity of the obtained approximate values $D_{L,N}(\omega_k)$ for each $\gamma$ and combine this with a simple bisection scheme to obtain an estimate of the minimum required value $\gamma_{\min}$ for which $\hatkt$ remains non-negative on the
chosen grid. As illustrated in Figure \ref{fig:gammamin}, we observe that $\gamma_{\min}$ remains close to its lower bound $\delta=1$ for $\nu,\lambda < 1$, and shows approximately bilinear growth for larger $\nu,\lambda$. In other words, the continuation process requires a significantly
larger domain as smoothness or correlation length increase. This also implies that the KL or wavelet frames
of the RKHS obtained in \S \ref{sec:kl} and \S \ref{sec:wavelets} become more redundant as these parameters increase.

\begin{figure}
		\centering
	\includegraphics[width=9cm]{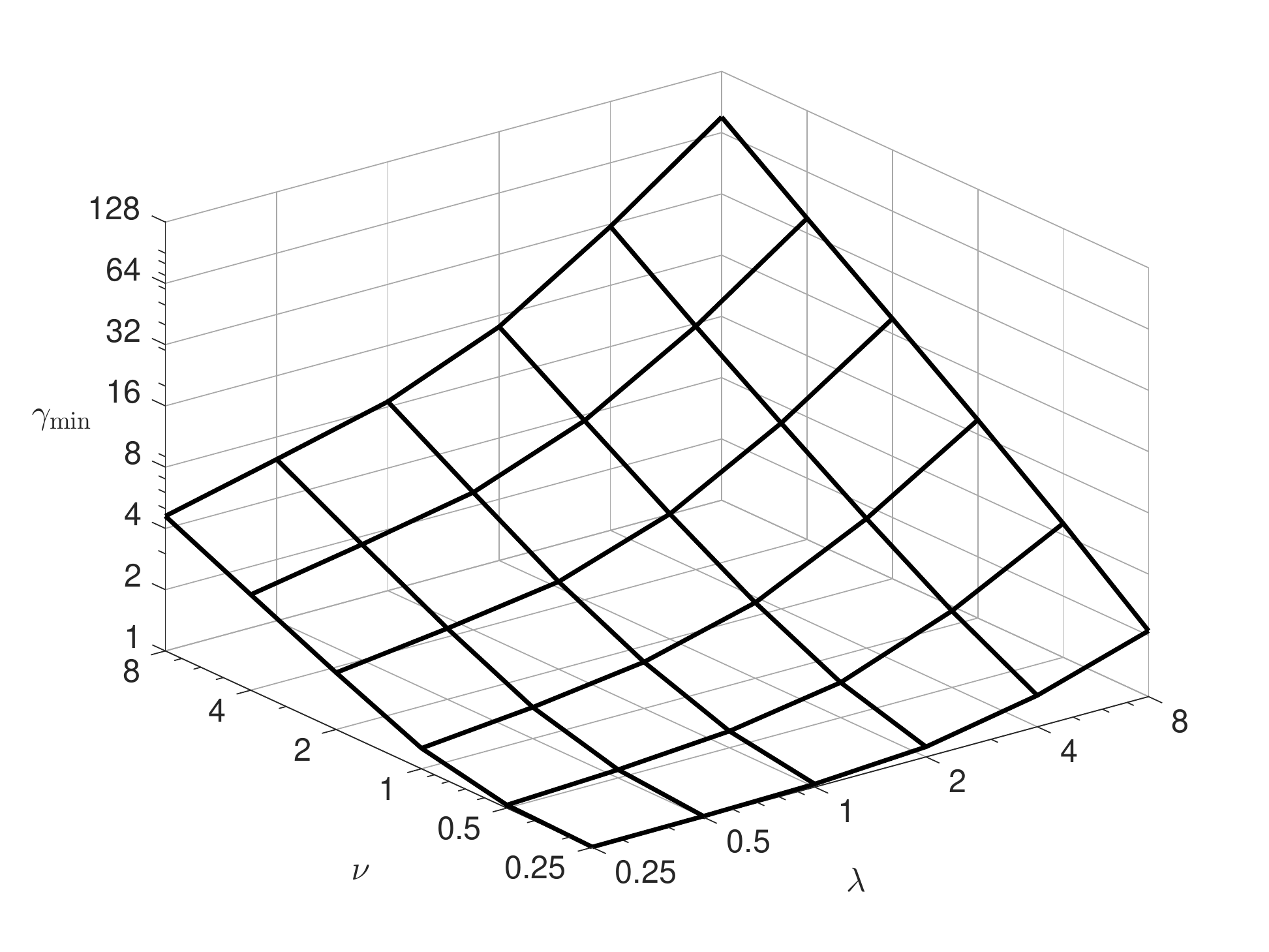}
	\caption{Numerically observed minimum value of $\gamma$ required for positivity of $\hatkt$, with $k$ as in \eqref{matern}, in dependence on the Mat\'ern parameters $\lambda,\nu$.}
	\label{fig:gammamin}
\end{figure}

\subsection{Mat\'ern wavelets}

The Meyer scaling function and wavelet can be defined by first
taking $\hat \vp$ to be a non-negative function such that
\be
|\hat \vp(\omega)|^2=\beta(\omega),
\ee
where $\beta(\omega)$ is a smooth and even function supported in $[-\frac{4\pi}3,\frac{4\pi}3]$,
such that 
\be
\beta(\omega)=1, \quad \omega\in \left[-\frac{2\pi}3, \frac{2\pi}3\right],
\ee
and
\be
\beta(\pi-\omega)+\beta(\pi+\omega)=1, \quad \omega\in \left[0, \frac\pi3 \right].
\ee
Then, one defines $\hat \psi$ by
\be
 \hat\psi(\omega) :=\bigl(\beta(\omega/2)-\beta(\omega)\bigr)^{1/2}e^{i\omega/2}.
\ee
One simple example with explicit expressions of $\hat \vp$ and $\hat \psi$, following the construction given in \cite{Dau}, is
\be\label{meyerscdef}
   \hat\varphi(\omega) := \begin{cases} 
 	 1 ,& \abs{\omega} \leq \frac{2\pi}3,\\
 	 \cos\Bigl(\frac\pi2 \nu\Bigl( \frac{3\abs{\omega}}{2\pi} - 1 \Bigr) \Bigr), & \frac{2\pi}3 < \abs{\omega} < \frac{4\pi}3,\\
 	 0, &\text{otherwise,}
 \end{cases}
\ee
and
\be
 \hat\psi(\omega) := \begin{cases}
 	  \sin\Bigl( \frac\pi2 \nu \Bigl( \frac{3\abs{\omega}}{2\pi} - 1\Bigr) \Bigr) e^{i\omega/2}, & \frac{2\pi}3 < \abs{\omega} \leq \frac{4\pi}3, \\
 	 \cos\Bigl( \frac\pi2 \nu \Bigl( \frac{3\abs{\omega}}{4\pi} - 1\Bigr) \Bigr) e^{i\omega/2}, & \frac{4\pi}3 < \abs{\omega} \leq \frac{8\pi}3, \\
 	 0, & \text{otherwise,} 
 \end{cases}
\ee
where we take
\be
 \nu(x) := \frac{\theta(x)}{\theta(x) + \theta(1-x)}
\ee
with $\theta$ given by \iref{theta}.

From these we now construct the one-dimensional versions of $\Phi^{\rm p}$ and $\Psi^{\rm p}$ described in \S\ref{sec:wavelets}. Recall that for the scaling function, $\Phi^{\rm p} = (2\gamma)^{-1/2} \sqrt{ \hatkt(0) }$, which we can directly approximate using \eqref{trapezoidalapprox}. For the wavelets, it suffices to consider $\Psi^{\rm p}_{\varepsilon,\ell,0}$ for each wavelet-type $\varepsilon$ and scale level $\ell$, 
since all further wavelets are obtained as translates of these functions. In the present univariate case, there is only a single wavelet type $\varepsilon=1$, 
and we omit the corresponding subscript in what follows.

By \eqref{gjbar} and \eqref{cnkp}, we have
\be
\Psi^{\rm p}_{\ell,0}(x) = \frac{1}{2\gamma} \sum_{n\in\Z} \sqrt{ \hatkt\left(\frac\pi\gamma n\right)} \,\Bigl( \sqrt{2^{1-\ell}\gamma}\, \hat\psi(2^{-\ell + 1}\pi n)\Bigr)\,e^{i \frac\pi\gamma n x}.
\ee
We now choose $L$ in \eqref{trapezoidalapprox} as $L=2 \gamma$, that is, $\omega_k = \frac{\pi k}{2 \gamma}$. We assume that $\gamma$ and $N = 2^J$, with $J>1$, are sufficiently large to ensure $D_{L,N}(\omega_k)\geq 0$ for the range of $k$ in \eqref{omegakdef}. This allows us to approximate the above expression by
\be
\label{approxwavelets}
  \tilde\Psi^{\rm p}_{\ell,0}(x) := \frac{1}{\sqrt{2^{\ell+1}\gamma}} \sum_{n = -N/4}^{ N/4-1} \sqrt{  D_{L,N}(\omega_{2 n})} \,\hat\psi(2^{-\ell + 1}\pi n)\,e^{i \frac\pi\gamma n x}.
\ee
Using the compact support of $\hat\psi$ and that $\abs{\hat\psi} \leq 1$, we obtain
\begin{multline}\label{psierrest}
  \abs{\Psi^{\rm p}_{\ell,0}(x) -  \tilde\Psi^{\rm p}_{\ell,0}(x)}
   \leq C 2^{-\ell/2} \biggl\{ \sum_{\substack{|n|\geq N/4 \\ \frac13 2^\ell < \abs{n} < \frac43 2^\ell }}\sqrt{ \hatkt(\pi\gamma^{-1} n)} \\
    +  \sum_{\substack{n \in\{ - N/4, \ldots,  N/4-1 \} \\  \frac13 2^\ell < \abs{n} < \frac43 2^\ell  }} \biggabs{\sqrt{ \hatkt(\pi\gamma^{-1} n)} - \sqrt{  D_{L,N}(\omega_{2 n})} }\biggr\}.
\end{multline}
Recall by \iref{framekt} we have $c (1+\abs{\omega})^{-2\nu-1}\leq \hatkt(\omega) \leq C(1+\abs{\omega})^{-2\nu-1}$.
It follows that the first sum can be bounded according to 
\be
\sum_{\substack{|n|\geq N/4 \\ \frac13 2^\ell < \abs{n} < \frac43 2^\ell }}\sqrt{ \hatkt(\pi\gamma^{-1} n)}\leq C \min\{ N^{-\nu + \frac12}, 2^\ell N^{-\nu-\frac12}\}.
\ee
For the second sum, we combine \iref{estimkt} with 
\be
 \biggabs{\sqrt{ \hatkt(\pi\gamma^{-1} n)} - \sqrt{  D_{L,N}(\omega_{2 n})} }
  \leq  \frac{\bigabs{{ \hatkt(\pi\gamma^{-1} n)} - {  D_{L,N}(\omega_{2 n})} }}{\sqrt{ \hatkt(\pi\gamma^{-1} n)}},
\ee
to obtain
\be
\sum_{\substack{n \in\{ - N/4, \ldots,  N/4-1 \} \\  \frac13 2^\ell < \abs{n} < \frac43 2^\ell  }} \biggabs{\sqrt{ \hatkt(\pi\gamma^{-1} n)} - \sqrt{  D_{L,N}(\omega_{2 n})} }
\leq CN^{-2\nu-1} \min\{ N^{\nu+\frac32}, 2^\ell N^{\nu+\frac12}\}.
\ee
This yields
\be
  \abs{\Psi^{\rm p}_{\ell,0}(x) -  \tilde\Psi^{\rm p}_{\ell,0}(x)} 
    \leq C 2^{-\ell/2}  \min\{ N^{-\nu + \frac12}, 2^\ell N^{-\nu-\frac12}\} =C  2^{-\nu J - \frac12 \abs{J-\ell}},
    \label{waveletuniferr}
    \ee
where $C>0$ depends on $\nu,\lambda,\gamma$, but not on $J$ or $\ell$.

The sum in \eqref{approxwavelets} can be evaluated by FFT simultaneously for the $2^{J-1}$ arguments
\be
 x = \frac{4\gamma k}{N}, \quad k=-\frac{N}4,\ldots, \frac{N}4-1,
\ee
at cost of order $J 2^J$. In other words, if we prescribe a grid size $h\sim 2^{-J}$, as a consequence of \eqref{waveletuniferr} we can determine the values of the wavelets 
at any level $\ell$ at the grid points up to an error of order $h^\nu$, using $h^{-1} \abs{\log h}$ operations in total. Since the wavelets are trigonometric polynomials, their values between grid points can be approximated with similar order of accuracy by local polynomial interpolation of neighboring grid values.

As an illustration, we display in Figures \ref{fig:wv1} and \ref{fig:wv2}
the obtained wavelets at scales $\ell=0,\dots,5$ for $\lambda=1$ and $(\nu,\gamma)=(1/2,3/2)$, $(\nu,\gamma)=(4,5)$.
Note that these wavelets behave asymptotically similarly to standard wavelets in terms of scale invariance.
As expected the size decay in scale depends on $\nu$.

\begin{figure}
		\centering
	\includegraphics[width=14cm]{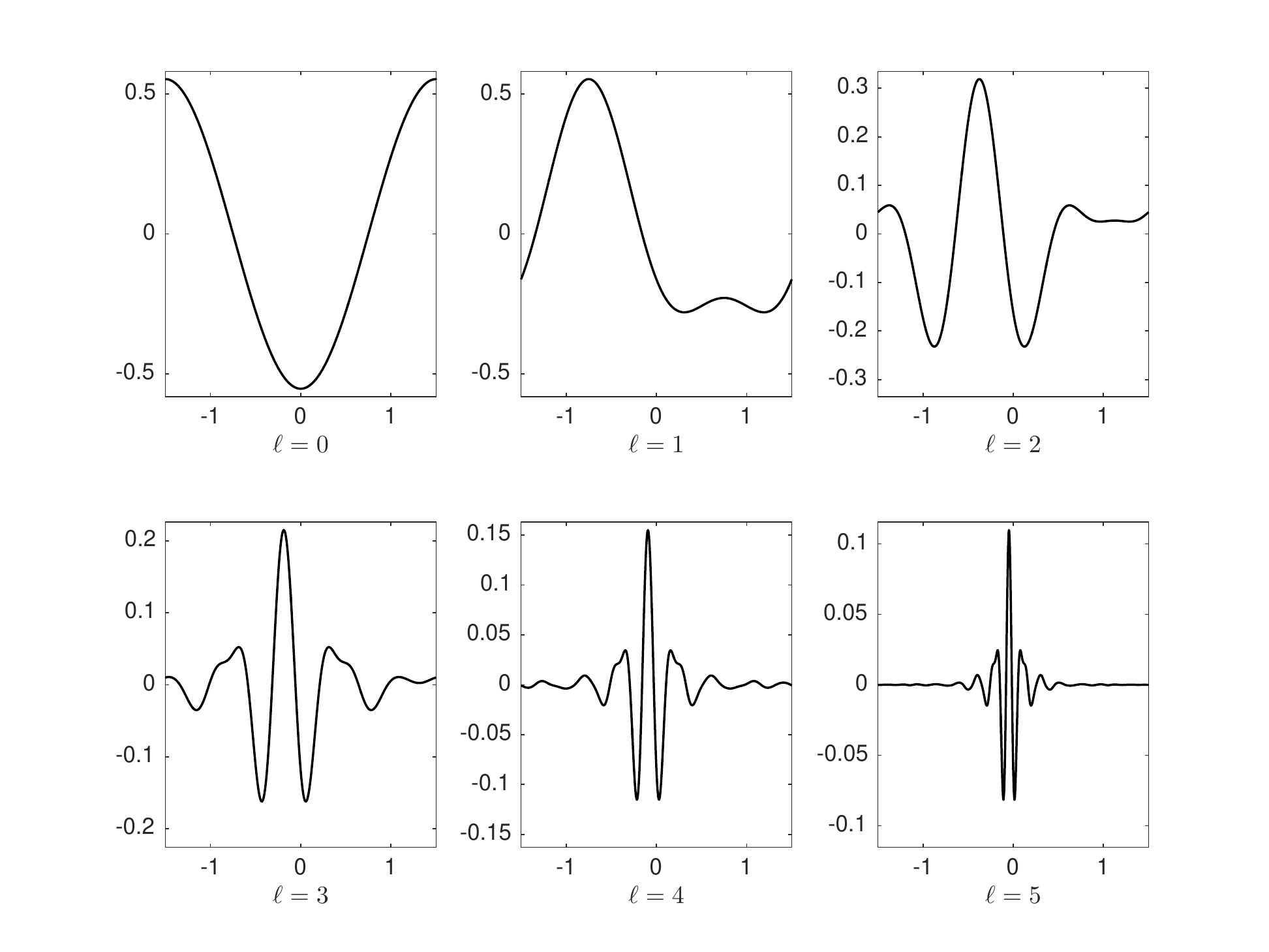}
	\caption{Wavelets $\Psi^{\rm p}_{\ell,0}$ obtained with $\lambda =1$, $\nu=\frac12$, and $\gamma=\frac32$, for $\ell=0,\ldots,5$.}
	\label{fig:wv1}
\end{figure}

\begin{figure}
		\centering
	\includegraphics[width=14cm]{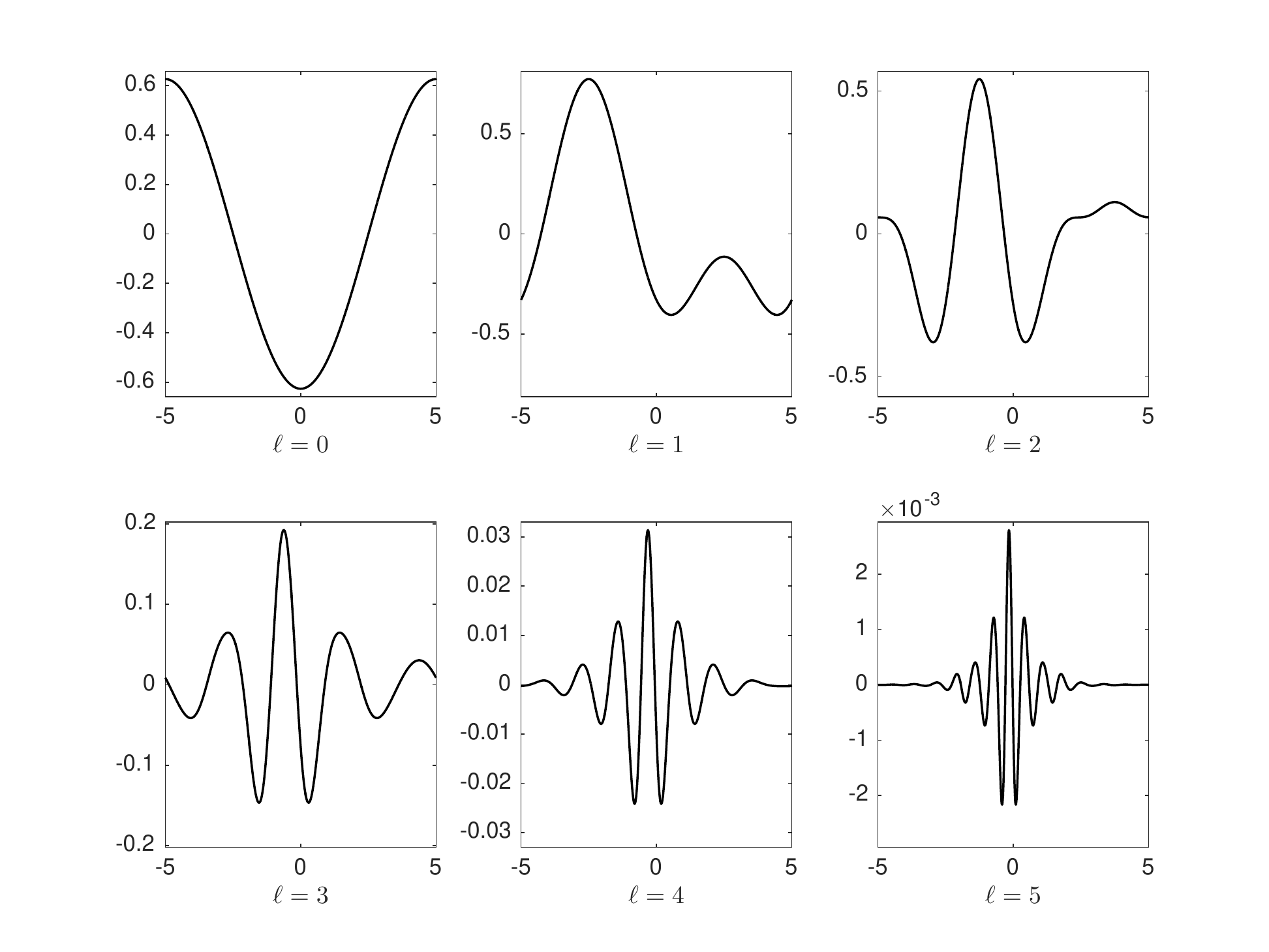}
	\caption{Wavelets $\Psi^{\rm p}_{\ell,0}$ obtained with $\lambda =1$, $\nu=4$, and $\gamma=5$, for $\ell=0,\ldots,5$.}
	\label{fig:wv2}
\end{figure}

\end{document}